\documentclass[12pt]{article}
\usepackage{a4}
\usepackage{amsthm}
\usepackage{amsfonts}
\usepackage{amssymb}
\usepackage{amsmath}
\usepackage{cite}
\usepackage{epsfig}

\newtheorem{theorem}{Theorem}
\newtheorem{corollary}[theorem]{Corollary}
\newtheorem{lemma}[theorem]{Lemma}
\newtheorem{proposition}[theorem]{Proposition}

\newtheorem{problem}{Problem}
\newtheorem*{claim*}{Claim}

\newcommand\NN{{\mathbb N}}

\DeclareTextCompositeCommand{\v}{OT1}{l}{l\nobreak\hspace{-.1em}'}
\DeclareTextCompositeCommand{\v}{OT1}{t}{t\nobreak\hspace{-.1em}'\nobreak\hspace{-.15em}}

\begin{document}
\title{Hypergraphs with minimum positive uniform Tur\'an density\thanks{This work has received funding from the European Research Council (ERC) under the European Union's Horizon 2020 research and innovation programme (grant agreement No 648509). This publication reflects only its authors' view; the European Research Council Executive Agency is not responsible for any use that may be made of the information it contains. The authors were also supported by the MUNI Award in Science and Humanities of the Grant Agency of Masaryk University.\newline An extended abstract containing the results presented in this paper has appeared in the proceedings of EuroComb'21.}}

\author{Frederik Garbe\thanks{Faculty of Informatics, Masaryk University, Botanick\'a 68A, 602 00 Brno, Czech Republic. E-mail: {\tt \{garbe,dkral,lamaison\}@fi.muni.cz}.}\and
        \newcounter{lth}
	\setcounter{lth}{2}
        Daniel Kr{\'a}\v{l}$^\fnsymbol{lth}$\and
	Ander Lamaison$^\fnsymbol{lth}$}
\date{} 
\maketitle

\begin{abstract}
Reiher, R\"odl and Schacht showed that
the uniform Tur\'an density of every $3$-uniform hypergraph is either $0$ or at least $1/27$, and
asked whether there exist $3$-uniform hypergraphs with uniform Tur\'an density equal or arbitrarily close to $1/27$.
We construct $3$-uniform hypergraphs with uniform Tur\'an density equal to $1/27$.
\end{abstract}

\section{Introduction}
\label{sec:intro}

Determining the minimum density of a (large) combinatorial structure required to contain a given (small) substructure
is a classical extremal combinatorics problem,
which can be traced to the work of Mantel~\cite{Man07} and Tur\'an~\cite{Tur41} in the first half of the 20th century.
The \emph{Tur\'an density} of a $k$-uniform hypergraph $H$,
which is denoted by $\pi(H)$, is the infimum over all $d$ such that
every sufficiently large host $k$-uniform hypergraph with edge density at least $d$ contains $H$ as a subhypergraph.
It can be shown~\cite{KatNS64} that
the Tur\'an density of $H$ is equal to the limit of the maximum density of a $k$-uniform $n$-vertex $H$-free hypergraph
($n$ tends to infinity);
in particular, Katona, Nemetz and Simonovits~\cite{KatNS64} showed that
this sequence of maximum densities is non-increasing and so the limit always exists.

The Tur\'an density of a complete graph $K_r$ of order $r$ is equal to $\frac{r-2}{r-1}$ as determined by Tur\'an~\cite{Tur41} himself.
Erd\H os and Stone~\cite{ErdS46} showed that
the Tur\'an density of any $r$-chromatic graph $H$ is equal to $\frac{r-2}{r-1}$, also see~\cite{ErdS66}.
The situation is more complex already for $3$-uniform hypergraphs, which we will call $3$-graphs for simplicity, compared to graphs (which are $2$-uniform hypergraphs).
In particular,
determining the Tur\'an density of the complete $4$-vertex $3$-graph $K_4^{(3)}$ is a major open problem, and
likewise determining the Tur\'an density of $K_4^{(3)-}$, defined as $K_4^{(3)}$ with an edge removed,
is a challenging open problem~\cite{BabT11,FraF84,Raz10}
despite some recent progress obtained using the flag algebra method of Razborov~\cite{Raz07};
also see the survey~\cite{Kee11} for further details.

It is well-known that
$H$-free graphs with density close to the Tur\'an density $\pi(H)$ are close to $(r-1)$-partite complete graphs~\cite{Fur15,Sim66},
i.e., the edges in such graphs are distributed in a highly non-uniform way.
The same applies to conjectured extremal constructions in the setting of $3$-graphs~\cite{FraF84}.
In this paper, we study the notion of uniform Tur\'an density of hypergraphs,
which requires the edges in the host hypergraph to be distributed uniformly.
This notion was suggested by Erd\H os and S\'os~\cite{ErdS82,Erd90} in the 1980s and
there is a large amount of recent progress in relation to this notion and to some of its variants~\cite{GleKV16,ReiRS16,ReiRS18a,ReiRS18b,ReiRS18c,ReiRS18}, see also the survey~\cite{Rei20}.
For example, Glebov, Volec and the second author~\cite{GleKV16} and Reiher, R\"odl and Schacht~\cite{ReiRS18a}
answered a question raised by Erd\H os and S\'os by showing that the uniform Tur\'an density of $K_4^{(3)-}$ is equal to $1/4$.

The following result of Reiher, R\"odl and Schacht~\cite{ReiRS18} is the starting point of our work:
\emph{the uniform Tur\'an density of every $3$-graph is either zero or at least $1/27$}.
Reiher et al.~\cite{ReiRS18} asked
whether there exist $3$-graphs with uniform Tur\'an density equal or arbitrarily close to $1/27$.
We answer this question in affirmative
by giving a sufficient condition for a $3$-graph to have uniform Tur\'an density equal to $1/27$ and
finding examples of $3$-graphs satisfying this condition.

We next introduce the notation needed to state our results precisely.
The \emph{$\varepsilon$-linear} density of an $n$-vertex hypergraph $H$
is the minimum density of an induced subhypergraph of $H$ with at least $\varepsilon n$ vertices.
The \emph{uniform Tur\'an density} of a hypergraph $H_0$ is the infimum over all $d$ such that
there exists $\varepsilon>0$ such that every sufficiently large hypergraph $H$ with $\varepsilon$-linear density $d$ contains $H_0$.
We also present an equivalent definition,
which is used by Reiher, R\"odl and Schacht~\cite{ReiRS16,ReiRS18a,ReiRS18b,ReiRS18c,ReiRS18}.
An $n$-vertex $k$-uniform hypergraph $H$ is \emph{$(d,\varepsilon)$-dense}
if every subset $W$ of its vertices induces at least $d\binom{|W|}{k}-\varepsilon n^k$ edges.
The uniform Tur\'an density of a hypergraph $H_0$ is the supremum over all $d$ such that
for every $\varepsilon>0$, there exist arbitrarily large $H_0$-free $(d,\varepsilon)$-dense hypergraphs.
It is easy to show that the two definitions are equivalent.

The notion of the uniform Tur\'an density is trivial for graphs as
the uniform Tur\'an density of every graph is equal to zero.
However, the situation is much more complex already for $3$-graphs.
As we have already mentioned,
the uniform Tur\'an density of $K_4^{(3)-}$ has been determined only recently~\cite{GleKV16,ReiRS18a}, and
the only other $3$-graphs with a positive uniform Tur\'an density that has been determined
are tight $3$-uniform cycles of length not divisible by three~\cite{BucCKMM} - note that for tight cycles divisible by $3$ the uniform Turán density is equal to $0$.
In particular,
determining the uniform Tur\'an density of $K_4^{(3)}$ is a challenging open problem
though it is believed that the 35-year-old construction of R\"odl~\cite{Rod86} showing that
the uniform Tur\'an density of $K_4^{(3)}$ is at least $1/2$ is optimal~\cite{Rei20}.

\begin{figure}
\begin{center}
\epsfbox{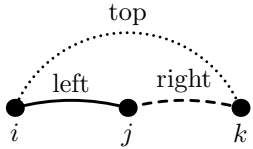}
\end{center}
\caption{Illustration of left, right and top pairs in an edge of a $3$-graph with ordered vertex set $i<j<k$.
	 The left pair is drawn solid, the right pair dashed and the top pair dotted
	 following the convention used later in Figure~\ref{fig:example9b}.}
\label{fig:LRT}
\end{figure}

Reiher, R\"odl and Schacht~\cite{ReiRS18} gave a characterization of $3$-graphs
with uniform Tur\'an density equal to zero,
which we now present.
Let $H$ be a $3$-graph with $n$ vertices.
We say that an ordering $v_1,\ldots,v_n$ of the vertices of $H$ is \emph{vanishing}
if the set of pairs $(i,j)$, $1\le i<j\le n$, can be partitioned to sets $L$, $T$ and $R$ such that
every edge $\{v_i,v_j,v_k\}$ of $H$, where $i<j<k$, satisfies that $(i,j)\in L$, $(i,k)\in T$ and $(j,k)\in R$.
The pairs that belong to $L$, $T$ and $R$ are referred to as \emph{left}, \emph{top} and \emph{right}, respectively (the reason
for this terminology comes from the visualization of the pairs of a triple by arcs over a horizontal line as in Figure~\ref{fig:LRT}).
We remark that when the vanishing ordering is fixed, the partition to the sets $L$, $T$ and $R$
is unique up to the pairs $(i,j)$ such that the vertices $v_i$ and $v_j$ are not contained in a common edge, and
we can choose all such undetermined pairs to be, say, left.
So, we can speak about left, top and right pairs whenever a vanishing ordering is fixed.

The characterization of $3$-graphs with uniform Tur\'an density equal to zero
reads as follows.

\begin{theorem}[Reiher, R\"odl and Schacht~\cite{ReiRS18}]
\label{thm:rrs}
Let $H$ be a $3$-graph.
The uniform Tur\'an density of $H$ is zero if and only if
$H$ has a vanishing ordering of its vertices.
\end{theorem}

If a $3$-graph $H$ has no vanishing ordering,
then the uniform Tur\'an density of $H$ is at least $1/27$ because of the following construction from~\cite{ReiRS18}.
Indeed, fix a $3$-graph $H$ with no vanishing ordering and
construct a random $n$-vertex $3$-graph $H_n$ as follows:
let $v_1,\ldots,v_n$ be the vertices of $H_n$,
randomly partition all pairs of those vertices to sets $L$, $T$ and $R$, and
include $\{v_i,v_j,v_k\}$, $1\le i<j<k\le n$, as an edge of $H_n$
if $(i,j)\in L$, $(i,k)\in T$ and $(j,k)\in R$.
Observe that $H$ cannot be a subhypergraph of $H_n$ (as $H$ has no vanishing ordering).
On the other hand, for every $\varepsilon>0$ and $\delta>0$,
there exists $n_0$ such that the density of every subset of at least $\varepsilon n$ vertices of $H_n$ for $n\ge n_0$
is at least $1/27-\delta$ with positive probability.
It follows that the uniform Tur\'an density of $H$ is at least $1/27$ as claimed.
Hence, Theorem~\ref{thm:rrs} implies the following.

\begin{corollary}\label{cor:rrs}
The uniform Tur\'an density of every $3$-graph is either zero or at least $1/27$.
\end{corollary}

Reiher, R\"odl and Schacht~\cite{ReiRS18} asked
whether there exist $3$-graphs with uniform Tur\'an density equal or arbitrarily close to $1/27$.
In this paper we answer this question in the affirmative.

\begin{theorem}\label{thm:intro}
There exists an infinite family of $3$-graphs with uniform Tur\'an density equal to $1/27$.
\end{theorem}

Theorem~\ref{thm:intro} is implied by the following.
In Theorem~\ref{thm:main}, we give a sufficient condition for a $3$-graph to have uniform Tur\'an density equal to $1/27$,
we then present a $7$-vertex $9$-edge $3$-graph (Theorem~\ref{thm:example9}) and
an infinite family of $3$-graphs (Theorem~\ref{thm:example8}),
whose smallest element has $8$ vertices and $9$ edges,
that satisfy the condition given in Theorem~\ref{thm:main}.
We remark that we have verified by a computer that there is no such $3$-graph with six or fewer vertices;
in fact, we have been able to show that every $3$-graph with six or fewer vertices
has Tur\'an density either equal to zero or at least $1/8$.

\section{Notation}
\label{sec:notation}

In this section, we introduce the notation used throughout the paper.
We write $[n]$ for the set of the first $n$ positive integers, i.e., $[n]=\{1,\ldots,n\}$.
An \emph{$n$-partitioned hypergraph $H$} is a $3$-graph such that
its vertex set is partitioned to sets $V_{ij}$, $1\le i<j\le n$, and
every edge $e$ of $H$ satisfies that there exist indices $1\le i<j<k\le n$ such that
one vertex of $e$ belongs to $V_{ij}$, one to $V_{ik}$ and one to $V_{jk}$.
The set of all edges of $H$ with a vertex from $V_{ij}$, one from $V_{ik}$ and one from $V_{jk}$
is called \emph{$(i,j,k)$-triad}, and
the \emph{density} of an $(i,j,k)$-triad is the number of edges forming the triad divided by $|V_{ij}|\cdot|V_{ik}|\cdot|V_{jk}|$.
Finally, the \emph{density} of an $n$-partitioned hypergraph $H$ is the minimum density of a triad of $H$.

We will use the following convention to simplify our notation used throughout the paper.
If $H$ is an $n$-partitioned hypergraph, we write $V_{ij}$, $1\le i<j\le n$, for its vertex parts, and
if $H'$ is an $n'$-partitioned hypergraph, we write $V'_{ij}$, $1\le i<j\le n'$, for its vertex parts,
i.e., we use the same mathematical accents when denoting a hypergraph as we do for its vertex parts without specifying the relation explicitly.
The \emph{reverse} of an $n$-partitioned hypergraph $H$ is an $n$-partitioned hypergraph $H'$ 
with the same vertex set and the same edge set as $H$
but with the partition of vertices given by $V'_{ij}=V_{n-j+1,n-i+1}$ for $1\le i<j\le n$.

Let $H$ be an $n$-partitioned hypergraph.
We say that $H'$ is an \emph{induced} subhypergraph of $H$
if there exists $I\subseteq [n]$ such that
$H'$ is an $|I|$-partitioned hypergraph,
its vertex parts are the parts $V_{ij}$ of $H$ such that $i,j\in I$ and
$H'$ contains all edges of $H$ with vertices in the vertex parts forming $H'$.
We will refer to $H'$ as the subhypergraph of $H$ \emph{induced} by the index set $I$.

We next define several notions of a normalized degree of a vertex of an $n$-partitioned hypergraph $H$.
Fix $1\le i<j<k\le n$ and define
\begin{itemize}
\item $d_{ij\to k}(v)$ for $v\in V_{ij}$ to be the number of edges of the $(i,j,k)$-triad containing $v$ divided by $|V_{ik}|\cdot|V_{jk}|$,
\item $d_{ik\to j}(v)$ for $v\in V_{ik}$ to be the number of edges of the $(i,j,k)$-triad containing $v$ divided by $|V_{ij}|\cdot|V_{jk}|$, and
\item $d_{jk\to i}(v)$ for $v\in V_{jk}$ to be the number of edges of the $(i,j,k)$-triad containing $v$ divided by $|V_{ij}|\cdot|V_{ik}|$.
\end{itemize}
Note that the arrow in the notation indicates to which part of the triad $v$ belongs.
Further,
$d_{ij,ik}(v,v')$ for $v\in V_{ij}$ and $v'\in V_{ik}$
is the number of edges of the $(i,j,k)$-triad containing $v$ and $v'$ divided by $|V_{jk}|$;
we analogously use $d_{ij,jk}(v,v')$ for $v\in V_{ij}$ and $v'\in V_{jk}$ and
$d_{ik,jk}(v,v')$ for $v\in V_{ik}$ and $v'\in V_{jk}$.
The considered hypergraph $H$ when using the just introduced notation will always be clear from the context,
so we decided not to include it as a part of the notation to keep the notation simpler.

An $N$-partitioned hypergraph $H$ \emph{embeds} an $n$-vertex hypergraph $H_0$
if it is possible to choose distinct $1\le a_1,\ldots,a_n\le N$ corresponding to the vertices of $H_0$ and
vertices $v_{ij}\in V_{a_ia_j}$ for $1\le i<j\le n$ such that
if the $i$-th, $j$-th and $k$-th vertex of $H_0$ form an edge,
then $\{v_{ij},v_{ik},v_{jk}\}$ is an edge of $H$.

In~\cite{Rei20},
Reiher gave a general theorem that
relates computing the uniform Tur\'an density of $3$-graphs to embeddings in partitioned hypergraphs.
In our notation, the theorem reads as follows.

\begin{proposition}[{Reiher~\cite[Theorem 3.3]{Rei20}}]
\label{prop:reiher}
Let $H$ be a $3$-graph and $d\in [0,1]$.
Suppose that for every $\delta>0$ there exists $N$ such that
every $N$-partitioned hypergraph with density at least $d+\delta$ embeds $H$.
Then, the uniform Tur\'an density of $H$ is at most $d$.
\end{proposition}

Some of our arguments use the classical Ramsey theorem for multicolored hypergraphs,
which we state below for reference.

\begin{theorem}[Ramsey~\cite{Ram30}]
\label{thm:ramsey}
For all $k,r,n\in\NN$, there exists $N\in\NN$ such that
every $k$-edge-coloring of a complete $R$-uniform hypergraph with $N$ vertices
contains a monochromatic complete $R$-uniform hypergraph with $n$ vertices.
\end{theorem}

\section{Preprocessing steps}
\label{sec:init}

In this section, we present two lemmas that
we use to tame a given partitioned hypergraph
before we can apply our main arguments.
The first lemma says that we can find a subhypergraph of a partitioned hypergraph such that
the proportions of left, top and right vertices with non-negligible degrees in all triads
are approximately the same.

\begin{lemma}
\label{lem:discretize}
For every $\varepsilon>0$ and $n\in\NN$,
there exists $N\in\NN$ such that the following holds.
For every $N$-partitioned hypergraph $H$,
there exist reals $\ell$, $t$ and $r$, and
an $n$-partitioned induced subhypergraph $H'$ such that
for all $1\le i<j<k\le n$
\begin{align*}
\ell |V'_{ij}| & \le |\{v\in V'_{ij},\,d_{ij\to k}(v)\ge\varepsilon\}| \le (\ell+\varepsilon) |V'_{ij}|, \\
t |V'_{ik}| & \le |\{v\in V'_{ik},\,d_{ik\to j}(v)\ge\varepsilon\}| \le (t+\varepsilon) |V'_{ik}|,\mbox{ and} \\
r |V'_{jk}| & \le |\{v\in V'_{jk},\,d_{jk\to i}(v)\ge\varepsilon\}| \le (r+\varepsilon) |V'_{jk}|.
\end{align*}
\end{lemma}

\begin{proof}
Apply Theorem~\ref{thm:ramsey} for $3$-graphs with $k_R=\left(\lfloor\varepsilon^{-1}\rfloor+1\right)^3$ and $n_R=n$ to get $N$ (the variables on the left of the equalities are named as in the statement of Theorem~\ref{thm:ramsey} but with the subscript $R$ added).
Let $H$ be an $N$-partitioned hypergraph.
Consider the following $k_R$-edge-coloring of the complete $3$-graph with vertex set $[N]$:
for $1\le i<j<k\le N$, let $L$ be the set of vertices $v\in V_{ij}$ such that $d_{ij\to k}(v)\ge\varepsilon$,
$T$ the set of $v\in V_{ik}$ such that $d_{ik\to j}(v)\ge\varepsilon$, and
$R$ the set of $v\in V_{jk}$ such that $d_{jk\to i}(v)\ge\varepsilon$, and
color the edge $\{i,j,k\}$ with the triple $\left(\left\lfloor\frac{|L|}{\varepsilon|V_{ij}|}\right\rfloor,\left\lfloor\frac{|T|}{\varepsilon|V_{ik}|}\right\rfloor,\left\lfloor\frac{|R|}{\varepsilon|V_{jk}|}\right\rfloor\right)$.
Theorem~\ref{thm:ramsey} implies that there exists a subset $I\subseteq [N]$ such that
all edges with vertices in $I$ have the same color, say $(\ell',t',r')$.
The $n$-partitioned subhypergraph $H'$ of $H$ induced by the index set $I$
satisfies the statement of the lemma with $\ell=\varepsilon\ell'$, $t=\varepsilon t'$ and $r=\varepsilon r'$.
\end{proof}

The next lemma concerns partitioned hypergraphs with density larger than $1/27$,
and yields that every such hypergraph contains an induced subhypergraph with one of the three properties described in the lemma.
We will refer to the first of these properties as
the case of \emph{horizontal intersection} and the other as the case of \emph{vertical intersection} (the second and third case
are symmetric by reversing the order of the parts of the partitioned hypergraph).
The case of horizontal intersection corresponds to the existence of an edge in an $(k',i,j)$-triad and
an edge in an $(i,j,k)$-triad, $k'<i<j<k$, that share a common vertex of $V_{ij}$ (the adjective horizontal
comes from the fact that the $(k',i,j)$-triad and the $(i,j,k)$-triad can be visualized by being drawn
as overlapping edges following each other on a horizontal line).
The case of vertical intersection corresponds to the existence of an edge in an $(i,k',j)$-triad and
an edge in an $(i,j,k)$, $i<k'<j<k$, that share a common vertex of $V_{ij}$ (the adjective vertical
comes from the fact that the two triads cannot be visualized as in the previous case as the edges are nested) or
the existence of an edge in an $(j,k',k)$-triad and an edge in an $(i,j,k)$, $i<j<k'<k$.
Using this terminology, the next lemma asserts that every partitioned hypergraph with density larger than $1/27$
has a subhypergraph such that all triads have non-trivial horizontal intersection or
all triads have non-trivial vertical intersection.
We will show that hypergraphs that we construct later can be embedded
in every partitioned hypergraph where all triads have non-trivial horizontal intersection and
in every partitioned hypergraph where all triads have non-trivial vertical intersection,
which are the two cases corresponding to the two possible outcomes of Lemma~\ref{lem:main}.

\begin{lemma}
\label{lem:density}
For every $\delta>0$, there exists $\varepsilon>0$ such that
for every $n\in\NN$, there exists $N\in\NN$ such that the following holds.
For every $N$-partitioned hypergraph $H$ with density at least $1/27+\delta$,
there exists an $n$-partitioned induced subhypergraph $H'$ of $H$ such that
at least one of the following holds.
\begin{itemize}
\item For all $1\le k'<i<j<k\le n$,
      the set $V'_{ij}$ contains at least $\varepsilon |V'_{ij}|$ vertices $v$ such that
      $d_{ij\to k}(v)\ge\varepsilon$ and $d_{ij\to k'}(v)\ge\varepsilon$.
\item For all $1\le i<k'<j<k\le n$,
      the set $V'_{ij}$ contains at least $\varepsilon |V'_{ij}|$ vertices $v$ such that
      $d_{ij\to k}(v)\ge\varepsilon$ and $d_{ij\to k'}(v)\ge\varepsilon$.
\item For all $1\le i<j<k'<k\le n$,
      the set $V'_{jk}$ contains at least $\varepsilon |V'_{jk}|$ vertices $v$ such that
      $d_{jk\to i}(v)\ge\varepsilon$ and $d_{jk\to k'}(v)\ge\varepsilon$.
\end{itemize}
\end{lemma}

\begin{proof}
We can assume that $\delta\le 1/2$ without loss of generality.
Set $\varepsilon=\delta/9$ and $\varepsilon_0=\delta/3$ and suppose that $n$ is given.
Apply Theorem~\ref{thm:ramsey} with $k_R=3$, $r_R=5$ and $n_R=2n+1$ to get $N_0$ and
apply Lemma~\ref{lem:discretize} with $\varepsilon_0$ and $N_0$ to get $N$.
Let $H$ be an $N$-partitioned hypergraph with density at least $1/27+\delta$, and
let $H_0$ be the induced $N_0$-partitioned subhypergraph provided by Lemma~\ref{lem:discretize} along with the reals $\ell$, $t$ and $r$ with the properties given in the statement of Lemma~\ref{lem:discretize}.

We first show the following

\begin{claim*}
$\ell+t+r\ge 1+\varepsilon_0$.
\end{claim*}

\begin{proof}[Proof of Claim]
Suppose that $\ell+t+r<1+\varepsilon_0$ and
choose arbitrary $i$, $j$ and $k$ such that $1\le i<j<k\le N_0$.
Let $L$ be the set of vertices $v\in V_{ij}$ such that $d_{ij\to k}(v)\ge\varepsilon$,
$T$ the set of vertices $v\in V_{ik}$ such that $d_{ik\to j}(v)\ge\varepsilon$, and
$R$ the set of vertices $v\in V_{jk}$ such that $d_{jk\to i}(v)\ge\varepsilon$.
Observe that
the number of edges of the $(i,j,k)$-triad that
contain a particular vertex $v\in V_{ij}\setminus L$ is at most $\varepsilon\cdot|V_{ik}|\cdot|V_ {jk}|$,
the number of edges that
contain a particular vertex $v\in V_{ik}\setminus T$ is at most $\varepsilon\cdot|V_{ij}|\cdot|V_{jk}|$, and
the number of edges that
contain a particular vertex $v\in V_{jk}\setminus R$ is at most $\varepsilon\cdot|V_{ij}|\cdot|V_{ik}|$.
Hence, the $(i,j,k)$-triad has at most $3\varepsilon\cdot|V_{ij}|\cdot|V_{jk}|\cdot|V_{ik}|$ edges
in addition to the edges with a vertex from $L$, a vertex from $T$ and a vertex from $R$.
The number of the edges of the latter type is at most $|L|\cdot |T|\cdot |R|$.
We derive using $\ell+t+r<1+\varepsilon_0$ that the density of the $(i,j,k)$-triad is at most
\[(\ell+\varepsilon_0)(t+\varepsilon_0)(r+\varepsilon_0)+3\varepsilon<
  \left(\frac{1+4\varepsilon_0}{3}\right)^3+3\varepsilon<
  \frac{1}{27}+3\varepsilon_0,
\]
which contradicts that the density of the $(i,j,k)$-triad is at least $\frac{1}{27}+\delta$.
\end{proof}

We next construct an auxiliary $3$-edge-coloring of the complete $5$-uniform hypergraph with vertex set $[N_0]$.
Consider $1\le k<i<k'<j<k''\le N_0$ and
let $R$ be the set of vertices $v$ of $V_{ij}$ such that $d_{ij\to k}(v)\ge\varepsilon$,
$T$ the set of vertices $v$ of $V_{ij}$ such that $d_{ij\to k'}(v)\ge\varepsilon$, and
$L$ the set of vertices $v$ of $V_{ij}$ such that $d_{ij\to k''}(v)\ge\varepsilon$.
If $|R\cap L|\ge\frac{\varepsilon_0}{3}|V_{ij}|$,
we color the edge $\{k,i,k',j,k''\}$ with the color red;
otherwise,
if $|L\cap T|\ge\frac{\varepsilon_0}{3}|V_{ij}|$,
we color the edge $\{k,i,k',j,k''\}$ with the color green;
otherwise,
if $|R\cap T|\ge\frac{\varepsilon_0}{3}|V_{ij}|$,
we color the edge $\{k,i,k',j,k''\}$ with the color blue.
If neither of the three cases applied, it would hold that
each of the sets $R\cap T$, $R\cap L$ and $L\cap T$ has fewer than $\frac{\varepsilon_0}{3}|V_{ij}|$ vertices;
this would imply that
\[|L\cup T\cup R|\ge |R|+|L|+|T|-|R\cap T|-|R\cap L|-|L\cap T|>(\ell+t+r-\varepsilon_0)|V_{ij}|\ge |V_{ij}|,\]
which is impossible since $L\cup T\cup R$ is a subset of $V_{ij}$.
Hence, one of the three cases always applies and so each edge gets a color.
Theorem~\ref{thm:ramsey} yields that
there exists a subset $I_0\subseteq [N_0]$ of size $2n+1$ such that
all edges with vertices from $I_0$ have the same color.

Let $I_0=\{a_1,a_2,a_3,\ldots,a_{2n+1}\}$ and
let $I=\{b_1,\ldots,b_n\}$ where $b_i=a_{2i}$ for $i=1,\ldots,n$.
We define the $n$-partitioned hypergraph $H'$ as the subhypergraph of $H_0$ induced by $I$,
where the vertex set $V_{ij}$ of $H'$ is identified with the vertex set $V_{b_ib_j}$ of $H_0$.
We claim that the $n$-partitioned hypergraph $H'$ has one of the three properties described in the statement of the lemma.
We distinguish three cases based on the common color of the edges of the complete $5$-uniform hypergraph induced by $I_0$.
If the common color is red,
we will show that the first property holds,
i.e., we obtain the case of the \emph{horizontal intersection}.
Indeed, for any integers $1\le k'<i<j<k\le n$,
consider $\{a_{2k'},a_{2i},a_{2i+1},a_{2j},a_{2k}\}$ and
the sets $L$ and $R$ from the definition of the color of this edge.
Observe that the set $L\cap R$ contains vertices $v$ such that
$d_{ij\to k'}(v)\ge\varepsilon$ and $d_{ij\to k}(v)\ge\varepsilon$ in the $n$-partitioned hypergraph $H'$,
which are vertices $v$ such that
$d_{a_{2i}a_{2j}\to a_{2k'}}(v)\ge\varepsilon$ and $d_{a_{2i}a_{2j}\to a_{2k}}(v)\ge\varepsilon$ in $H$.

If the common color is green,
we will show that the second property holds.
Indeed, for any integers $1\le i<k'<j<k\le n$,
consider the edge $\{a_{2i-1},a_{2i},a_{2k'},a_{2j},a_{2k}\}$.
The sets $T$ and $L$ from the definition of the color of the edge
have the property that the set $L\cap T$ contains vertices $v$ such that
$d_{ij\to k'}(v)\ge\varepsilon$ and $d_{ij\to k}(v)\ge\varepsilon$ in $H'$.
Finally, if the common color is blue,
we conclude using an argument analogous to the just analyzed case that $H'$ has the third property given in the statement of the lemma.
\end{proof}

\section{Embedding lemma}
\label{sec:embed}

In this section,
we prove Lemma~\ref{lem:main} which asserts that every partitioned hypergraph with density larger than $1/27$
contains one of two specific general substructures that can be used to embed our considered hypergraphs.
We remark that Lemmas~\ref{lem:Rleft}--\ref{lem:Rtop} are implicitly contained in~\cite{ReiRS18}
where they were proven using an iterative approach;
we prove them using Ramsey type arguments and
extend them to a more general setting (Lemmas~\ref{lem:Rhorizontal} and~\ref{lem:Rvertical})
which is needed to deal with two possible outcomes of Lemma~\ref{lem:density}.

We start with stating and proving Lemma~\ref{lem:Rleft}.

\begin{lemma}
\label{lem:Rleft}
For every $n$ and $\varepsilon>0$, there exists $N$ such that the following holds.
If $H$ is an $N$-partitioned hypergraph and
for each $1\le i<j<k\le N$ a subset $S_{ijk}$ of $V_{ij}$ with at least $\varepsilon |V_{ij}|$ vertices is given,
then there exist a subset $I\subseteq [N]$ of size $n$ and vertices $s_{ij}$, $i,j\in I$, $i<j$, such that
$s_{ij}\in S_{ijk}$ for all $i,j,k\in I$, $i<j<k$.
\end{lemma}

\begin{proof}
Apply Theorem~\ref{thm:ramsey} with $k_R=2$, $r_R=n$ and $n_R=\max\{2n,2+\lceil n/\varepsilon\rceil\}$ to get $N$ (the variables
on the left of the equalities are named as in the statement of Theorem~\ref{thm:ramsey} with the subscript $R$ added
to distinguish them from the variables in the statement of the lemma).
Let $H$ be an $N$-partitioned hypergraph and sets $S_{ijk}\subseteq V_{ij}$ be as described in the statement of the lemma.
We construct an auxiliary $2$-edge-coloring of the complete $n$-uniform hypergraph on the vertex set $[N]$ as follows:
an $n$-tuple $a_1<a_2<\cdots<a_n$ is colored blue
if the $n-2$ sets $S_{a_1a_2a_3},S_{a_1a_2a_4},\ldots,S_{a_1a_2a_n}$ have a common vertex, and
it is colored red otherwise.
By Theorem~\ref{thm:ramsey},
there exist $a_1,\ldots,a_{n_R}\in [N]$, $a_1<a_2<\cdots<a_{n_R}$, such that
all $n$-tuples of $a_1,\ldots,a_{n_R}$ have the same color.
We next distinguish two cases depending on the common color of those $n$-tuples.

If the common color of the $n$-tuples is blue,
then we set $I=\{a_1,\ldots,a_{n}\}$ and
let $s_{a_ia_j}$ for $1\le i<j\le n$ be any element contained in the intersection of
the sets $S_{a_ia_ja_{j+1}},S_{a_ia_ja_{j+2}},\ldots,S_{a_ia_ja_{n}}$.

Suppose that the common color for the $n$-tuples is red.
Since each of the sets $S_{a_1a_2a_\ell}$ for $\ell=3,\ldots,2+\lceil n/\varepsilon\rceil$
contains at least $\varepsilon |V_{a_1a_2}|$ elements of $V_{a_1a_2}$ and
the number of choices for $\ell$ is $\lceil n/\varepsilon\rceil$,
there exist an element $s\in V_{a_1a_2}$ and $J\subseteq\{a_3,\ldots,a_{2+\lceil n/\varepsilon\rceil}\}$, $|J|\ge n$, such that
$s\in S_{a_1a_2a}$ for every $a\in J$.
This implies that the $n$-tuple formed by $a_1$, $a_2$ and any $n-2$ elements of $J$ should be blue,
which contradicts that the common color for the $n$-tuples formed by $a_1,\ldots,a_{n_R}$ is red.
\end{proof}

The first of the next two lemmas follows from Lemma~\ref{lem:Rleft} by applying it to the reverse of $H$,
however, for later use it is beneficial to state it explicitly;
the proof of the second lemma follows along the lines of Lemma~\ref{lem:Rleft} and we only include its sketch for completeness.

\begin{lemma}
\label{lem:Rright}
For every $n$ and $\varepsilon>0$, there exists $N$ such that the following holds.
If $H$ is an $N$-partitioned hypergraph and
for each $1\le i<j<k\le N$ a subset $S_{ijk}$ of $V_{jk}$ with at least $\varepsilon |V_{jk}|$ vertices is given,
then there exist a subset $I\subseteq [N]$ of size $n$ and vertices $s_{jk}$, $j,k\in I$, $j<k$, such that
$s_{jk}\in S_{ijk}$ for all $i,j,k\in I$, $i<j<k$.
\end{lemma}

\begin{lemma}
\label{lem:Rtop}
For every $n$ and $\varepsilon>0$, there exists $N$ such that the following holds.
If $H$ is an $N$-partitioned hypergraph and
for each $1\le i<j<k\le N$ a subset $S_{ijk}$ of $V_{ik}$ with at least $\varepsilon |V_{ik}|$ vertices is given,
then there exist a subset $I\subseteq [N]$ of size $n$ and vertices $s_{ik}$, $i,k\in I$, $i<k$, such that
$s_{ik}\in S_{ijk}$ for all $i,j,k\in I$, $i<j<k$.
\end{lemma}

\begin{proof}
As we have mentioned, we only sketch the proof as it follows the lines of the proof of Lemma~\ref{lem:Rleft}.
We first apply Theorem~\ref{thm:ramsey} with $k_R=2$, $r_R=n$ and $n_R=\max\{n^2,2+\lceil n/\varepsilon\rceil\}$ to get $N$.
Suppose that an $N$-partitioned hypergraph $H$ and sets $S_{ijk}\subseteq V_{ik}$ are given.
We construct an auxiliary $2$-edge-coloring of the complete $n$-uniform hypergraph on the vertex set $[N]$ as follows:
an $n$-tuple $a_1<a_2<\cdots<a_n$ is colored blue
if the $n-2$ sets $S_{a_1a_2a_n},S_{a_1a_3a_n},\ldots,S_{a_1a_{n-1}a_n}$ have a common vertex, and
it is colored red otherwise.
By Theorem~\ref{thm:ramsey},
there exist $a_1,\ldots,a_{n_R}\in [N]$, $a_1<a_2<\cdots<a_{n_R}$, such that
all $n$-tuples of $a_1,\ldots,a_{n_R}$ have the same color.
If the common color is blue, we set $I=\{a_1,a_{n+1},\ldots,a_{n^2-n+1}\}$.
If the common color is red,
we argue as in the proof of Lemma~\ref{lem:Rleft} that
there is an element $s\in V_{a_1a_{\lceil n/\varepsilon\rceil+2}}$ contained
in at least $n$ sets $S_{a_1a_{\ell}a_{\lceil n/\varepsilon\rceil+2}}$
where $\ell$ ranges between $2$ and $\lceil n/\varepsilon\rceil+1$.
Hence, the $n$-tuple formed by $a_1$, $a_{\lceil n/\varepsilon\rceil+2}$ and $n-2$ choices of $\ell$ with this property
should be blue, which contradicts that the common color of the $n$-tuples formed by $a_1,\ldots,a_{n_R}$ is red.
\end{proof}

We now extend Lemmas~\ref{lem:Rleft}--\ref{lem:Rtop} to the setting needed to prove Lemma~\ref{lem:main}.

\begin{lemma}
\label{lem:Rhorizontal}
For every $n$ and $\varepsilon>0$, there exists $N$ such that the following holds.
If $H$ is an $N$-partitioned hypergraph and
for each $1\le i<j<k\le N$ subsets $S_{ijk}$ of $V_{ij}$ and $S'_{ijk}$ of $V_{jk}$ are given such that
the intersection $S'_{k'ij}\cap S_{ijk}$ has at least $\varepsilon |V_{ij}|$ elements for all $1\le k'<i<j<k\le N$,
then there exists a subset $I\subseteq [N]$ of size $n$ and vertices $s_{ij}$, $i,j\in I$, $i<j$, such that
$s_{ij}\in S'_{k'ij}$ and $s_{ij}\in S_{ijk}$ for all $k'<i<j<k$ with $i,j,k,k'\in I$.
\end{lemma}

\begin{proof}
Apply Theorem~\ref{thm:ramsey} with $k_R=2$, $r_R=2n-2$ and $n_R=\max\{3n,2+2\lceil n/\varepsilon\rceil\}$ to get $N$.
Let $H$ be an $N$-partitioned hypergraph and sets $S_{ijk},S'_{ijk}\subseteq V_{ij}$ as described in the statement of the lemma.
We construct an auxiliary $2$-edge-coloring of the complete $(2n-2)$-uniform hypergraph on the vertex set $[N]$ as follows:
an $(2n-2)$-tuple $a_1<a_2<\cdots<a_{2n-2}$ is colored blue
if the $n-2$ sets $S'_{a_1a_{n-1}a_n},S'_{a_2a_{n-1}a_n},\ldots,S'_{a_{n-2}a_{n-1}a_n}$ and
the $n-2$ sets $S_{a_{n-1}a_na_{n+1}},S_{a_{n-1}a_na_{n+2}},\ldots,S_{a_{n-1}a_na_{2n-2}}$
have a common vertex, and
it is colored red otherwise.
By Theorem~\ref{thm:ramsey},
there exist $a_1,\ldots,a_{n_R}\in [N]$, $a_1<a_2<\cdots<a_{n_R}$, such that
all $(2n-2)$-tuples of $a_1,\ldots,a_{n_R}$ have the same color.
We next distinguish two cases depending on the common color of these $(2n-2)$-tuples.

If the common color of the $(2n-2)$-tuples is blue,
we set $I=\{a_{n-1},\ldots,a_{2n-2}\}$ and
let $s_{a_ia_j}$ for $n-1\le i<j\le 2n-2$ be any element contained in the intersection of
the sets $S'_{a_{i-(n-2)}a_ia_j},\ldots,S'_{a_{i-1}a_ia_j}$ and $S_{a_ia_ja_{j+1}},\ldots,S_{a_ia_ja_{j+n-2}}$.

Suppose that the common color for the $(2n-2)$-tuples is red.
Since each of the $m:=\lceil n/\varepsilon\rceil$ many sets $S'_{a_{\ell}a_{m+1}a_{m+2}}\cap S_{a_{m+1}a_{m+2}a_{m+2+\ell}}$ for $\ell=1,\ldots,m$
contains at least $\varepsilon |V_{a_{m+1}a_{m+2}}|$ elements of $V_{a_{m+1}a_{m+2}}$,
there exist an element $s\in V_{a_{m+1}a_{m+2}}$ and $J\subseteq\{1,\ldots,m\}$, $|J|=n-2$, such that
$s\in S'_{a_{\ell}a_{m+1}a_{m+2}}\cap S_{a_{m+1}a_{m+2}a_{m+2+\ell}}$ for every $\ell\in J$,
i.e.,
$s\in S'_{a_{\ell}a_{m+1}a_{m+2}}$ and
$s\in S_{a_{m+1}a_{m+2}a_{m+2+\ell}}$ for every $\ell\in J$.
It follows that
the $(2n-2)$-tuple formed by the indices $a_{m+1}$, $a_{m+2}$,
$a_{\ell}$ and $a_{m+2+\ell}$, $\ell\in J$, should be colored blue.
This contradicts that the common color for the $(2n-2)$-tuples formed by elements of $I$ is red.
\end{proof}

The proof of the next lemma follows along the lines of the proof of Lemma~\ref{lem:Rhorizontal}
but since it is not completely analogous, we decided to include its sketch for completeness.

\begin{lemma}
\label{lem:Rvertical}
For every $n$ and $\varepsilon>0$, there exists $N$ such that the following holds.
If $H$ is an $N$-partitioned hypergraph and
for each $1\le i<j<k\le N$ subsets $S_{ijk}$ of $V_{ij}$ and $S'_{ijk}$ of $V_{ik}$ are given such that
the intersection $S'_{ik'j}\cap S_{ijk}$ has at least $\varepsilon |V_{ij}|$ elements for all $1\le i<k'<j<k\le N$,
then there exists a subset $I\subseteq [N]$ of size $n$ and vertices $s_{ij}$, $i,j\in I$, $i<j$, such that
$s_{ij}\in S'_{ik'j}$ and $s_{ij}\in S_{ijk}$ for all $i<k'<j<k$ with $i,j,k,k'\in I$.
\end{lemma}

\begin{proof}
First apply Theorem~\ref{thm:ramsey} with $k_R=2$, $r_R=2n-2$ and $n_R=\max\{n^2,2+2\lceil n/\varepsilon\rceil\}$ to get $N$.
Consider an $N$-partitioned hypergraph $H$ and sets $S_{ijk}$ and $S'_{ijk}$ as given in the statement.
We construct an auxiliary $2$-edge-coloring of the complete $(2n-2)$-uniform hypergraph on the vertex set $[N]$ as follows:
an $(2n-2)$-tuple $a_1<a_2<\cdots<a_{2n-2}$ is colored blue
if the $n-2$ sets $S'_{a_1a_2a_n},S'_{a_1a_3a_n},\ldots,S'_{a_1a_{n-1}a_n}$ and
the $n-2$ sets $S_{a_{n-1}a_na_{n+1}},S_{a_{n-1}a_na_{n+2}},\ldots,S_{a_{n-1}a_na_{2n-2}}$
have a common vertex, and
it is colored red otherwise.
By Theorem~\ref{thm:ramsey},
we get $a_1,\ldots,a_{n_R}\in [N]$, $a_1<a_2<\cdots<a_{n_R}$, such that 
all $(2n-2)$-tuples of $a_1,\ldots,a_{n_R}$ have the same color.
If the common color of the $(2n-2)$-tuples is blue,
we set $I=\{a_{1},a_{n+1},\ldots,a_{n^2-n+1}\}$;
the existence of $s_{ij}$ follows as all $(2n-2)$-tuples are blue.
Suppose that the common color for the $(2n-2)$-tuples is red.
Similarly to the proof of Lemma~\ref{lem:Rhorizontal},
we consider intersections $S'_{a_1a_{1+\ell}a_{\lceil n/\varepsilon\rceil+2}}\cap S_{a_1a_{\lceil n/\varepsilon\rceil+2}a_{\lceil n/\varepsilon\rceil+\ell+2}}$
where $\ell$ ranges between $1$ and $\lceil n/\varepsilon\rceil$ and
argue that there exist $n$ of these intersection that have a vertex in common;
this implies that one of $(2n-2)$-tuples should be blue.
\end{proof}

To prove Lemma~\ref{lem:main}, we need an additional auxiliary lemma.

\begin{lemma}
\label{lem:neighbors}
The following holds for every tripartite hypergraph $G$ with parts $A$, $B$ and $C$ and every $\varepsilon>0$.
If a vertex $a$ of $A$ is contained in at least $\varepsilon |B|\cdot |C|$ edges of $G$,
then there exist at least $\varepsilon |B|/2$ vertices $b$ of $B$ such that
$a$ and $b$ are contained together in at least $\varepsilon |C|/2$ edges of $G$.
\end{lemma}

\begin{proof}
Let $B'\subseteq B$ be the subset of vertices $b$ of $B$ which
are contained together with $a$ in at least $\varepsilon |C|/2$ edges of $G$.
If $|B'|<\varepsilon |B|/2$,
then there are less than $\varepsilon |B|\cdot |C|/2$ edges containing the vertex $a$ and a vertex $b\in B'$.
Since any vertex $b\in B\setminus B'$ is contained together with $a$ in less than $\varepsilon |C|/2$ edges of $G$,
the number of edges containing the vertex $a$ is less than $\varepsilon |B|\cdot |C|$,
which contradicts the assumption of the lemma.
\end{proof}

We are now ready to prove the embedding lemma,
which is the main result of this section.
The lemma will be used to upper bound the uniform Tur\'an density of hypergraphs constructed in the next section.

\begin{lemma}
\label{lem:main}
For every $\delta>0$ and $n\in\NN$,
there exists $N\in\NN$ such that the following holds.
For every $N$-partitioned hypergraph $H$ with density at least $1/27+\delta$,
there exists an $n$-partitioned induced subhypergraph $H'$ of $H$ and
vertices $\alpha_{ij},\beta_{ij},\gamma_{ij},\beta'_{ij},\gamma'_{ij}\in V_{ij}$ for all $1\le i<j\le n$ such that
$\{\alpha_{ij},\beta_{jk},\gamma_{ik}\}$ is an edge of $H'$ for all $1\le i<j<k\le n$ and
at least one of the following holds:
\begin{itemize}
\item For all $1\le i<j<k\le n$, $\{\beta_{ij},\beta'_{jk},\gamma'_{ik}\}$ is an edge of $H'$.
\item For all $1\le i<j<k\le n$, $\{\gamma_{ij},\beta'_{jk},\gamma'_{ik}\}$ is an edge of $H'$.
\item For all $1\le i<j<k\le n$, $\{\beta'_{ij},\gamma_{jk},\gamma'_{ik}\}$ is an edge of $H'$.
\end{itemize}
\end{lemma}

\begin{proof}
Fix $\delta>0$ and $n\in\NN$.
Apply Lemma~\ref{lem:density} with $\delta$ to get $\varepsilon>0$.
Apply Lemma~\ref{lem:Rleft} with $n$ and $\varepsilon/2$ to get $n_1$,
then Lemma~\ref{lem:Rtop} with $n_1$ and $\varepsilon/2$ to get $n_2$,
then Lemma~\ref{lem:Rright} with $n_2$ and $\varepsilon/2$ to get $n_3$,
then Lemma~\ref{lem:Rtop} with $n_3$ and $\varepsilon/2$ to get $n_4$, and
finally Lemma~\ref{lem:Rhorizontal} with $n_4+2$ and $\varepsilon$ to get $N_h$.
Next apply Lemma~\ref{lem:Rright} with $n_1$ and $\varepsilon/2$ to get $n'_2$,
then Lemma~\ref{lem:Rright} again with $n'_2$ and $\varepsilon/2$ to get $n'_3$,
then Lemma~\ref{lem:Rtop} with $n'_3$ and $\varepsilon/2$ to get $n'_4$, and
finally Lemma~\ref{lem:Rvertical} with $n'_4+2$ and $\varepsilon$ to get $N_v$.
We obtain $N$ by applying Lemma~\ref{lem:density} with $\max\{N_h,N_v\}$ (with $\delta$ and $\varepsilon$ as fixed earlier).

Let $H$ be an $N$-partitioned hypergraph with density at least $1/27+\delta$.
By Lemma~\ref{lem:density},
there exists a $\max\{N_h,N_v\}$-partitioned induced subhypergraph $H^5$ of $H$ that
satisfies one of the three properties given in the statement of Lemma~\ref{lem:density}.
We start with analyzing the case that the first property holds, i.e., the case of \emph{horizontal intersection};
this case results in the first case described in the statement of the lemma.
The sought hypergraph $H'$ and
the vertices $\alpha_{ij},\beta_{ij},\gamma_{ij},\beta'_{ij},\gamma'_{ij}$ are obtained as follows.
\begin{itemize}
\item For $1\le k'<i<j<k\le N_h$,
      let $S'_{k'ij}$ be the set of the vertices $v\in V_{ij}$ such that $d_{ij\to k'}(v)\ge\varepsilon$ and
      let $S_{ijk}$ be the set of the vertices $v\in V_{ij}$ such that $d_{ij\to k}(v)\ge\varepsilon$.
      By assumption we are in the horizontal intersection case and the first outcome of Lemma~\ref{lem:density} applies
      and therefore $|S'_{k'ij}\cap S_{ijk}|\ge\varepsilon |V_{ij}|$ for all $1\le k'<i<j<k\le N_h$.
      Hence, Lemma~\ref{lem:Rhorizontal} yields that
      there exists an $(n_4+2)$-partitioned induced subhypergraph of $H^5$ and vertices $\beta_{ij}$ such that
      $d_{ij\to k'}(\beta_{ij})\ge\varepsilon$ and $d_{ij\to k}(\beta_{ij})\ge\varepsilon$ for all $1\le k'<i<j<k\le n_4+2$;
      removing the first and the last part yields an $n_4$-partitioned induced subhypergraph $H^4$ of $H^5$ such that
      $d_{jk\to i}(\beta_{jk})\ge\varepsilon$ and $d_{ij\to k}(\beta_{ij})\ge\varepsilon$ for all $1\le i<j<k\le n_4$.
\item For $1\le i<j<k\le n_4$,
      let $S_{ijk}$ be the set of vertices $v\in V_{ik}$ such that $d_{ij,ik}(\beta_{ij},v)\ge\varepsilon/2$;
      observe that each of the sets $S_{ijk}$ contains at least $\varepsilon |V_{ik}|/2$ elements by Lemma~\ref{lem:neighbors}.
      Lemma~\ref{lem:Rtop} yields that
      there exists an $n_3$-partitioned induced subhypergraph $H^3$ of $H^4$ and vertices $\gamma'_{ik}$ such that
      $d_{ij,ik}(\beta_{ij},\gamma'_{ik})\ge\varepsilon/2$ for all $1\le i<j<k\le n_3$.
\item For $1\le i<j<k\le n_3$,
      let $S_{ijk}$ be the set of vertices $v\in V_{jk}$ that
      form an edge together with $\beta_{ij}$ and $\gamma'_{ik}$ in the $(i,j,k)$-triad of $H^3$, and
      apply Lemma~\ref{lem:Rright} to get an $n_2$-partitioned induced subhypergraph $H^2$ of $H^3$ and vertices $\beta'_{jk}$ such that
      $\{\beta_{ij},\beta'_{jk},\gamma'_{ik}\}$ is an edge of $H^2$ for all $1\le i<j<k\le n_2$.
\item For $1\le i<j<k\le n_2$,
      let $S_{ijk}$ be the set of vertices $v\in V_{ik}$ such that $d_{jk,ik}(\beta_{jk},v)\ge\varepsilon/2$;
      observe that each of the sets $S_{ijk}$ contains at least $\varepsilon |V_{ik}|/2$ elements by Lemma~\ref{lem:neighbors}.
      So, Lemma~\ref{lem:Rtop} yields that
      there exists an $n_1$-partitioned induced subhypergraph $H^1$ of $H^2$ and vertices $\gamma_{ik}$ such that
      $d_{jk,ik}(\beta_{jk},\gamma_{ik})\ge\varepsilon/2$ for all $1\le i<j<k\le n_1$.
\item For $1\le i<j<k\le n_1$,
      let $S_{ijk}$ be the set of vertices $v\in V_{ij}$ that form an edge with $\beta_{jk}$ and $\gamma_{ik}$ in $H^1$.
      Lemma~\ref{lem:Rleft} yields that
      there exists an $n$-partitioned induced subhypergraph $H'$ of $H^1$ and vertices $\alpha_{ij}$ such that
      $\{\alpha_{ij},\beta_{jk},\gamma_{ik}\}$ is an edge of $H'$ for $1\le i<j<k\le n$.
\end{itemize}
The hypergraph $H'$ together with the vertices $\alpha_{ij},\beta_{ij},\gamma_{ij},\beta'_{ij},\gamma'_{ij}$
satisfies the first case of the lemma.

The case of \emph{vertical intersection} from Lemma~\ref{lem:density} is analyzed in an analogous way.
We next sketch the steps resulting in the sought hypergraph $H'$ and
the vertices $\alpha_{ij},\beta_{ij},\gamma_{ij},\beta'_{ij},\gamma'_{ij}$
if the second property in the statement of Lemma~\ref{lem:density} applies.
\begin{itemize}
\item Lemma~\ref{lem:Rvertical} is used to obtain an $(n'_4+2)$-partitioned induced subhypergraph of $H^5$ and
      vertices $\gamma_{ij}$ such that 
      $d_{ij\to k}(\gamma_{ij})\ge\varepsilon$ and $d_{ij\to k'}(\gamma_{ij})\ge\varepsilon$ for all $1\le i<k'<j<k\le n_4+2$;
      removing the first and the last part yields an $n'_4$-partitioned induced subhypergraph $H^4$ of $H^5$.
\item Lemma~\ref{lem:Rtop} is used to obtain an $n'_3$-partitioned induced subhypergraph of $H^3$ of $H^4$ and
      vertices $\gamma'_{ik}$ such that
      $d_{ij,ik}(\gamma_{ij},\gamma'_{ik})\ge\varepsilon/2$ for all $1\le i<j<k\le n_3$.
\item Lemma~\ref{lem:Rright} is used to obtain an $n'_2$-partitioned induced subhypergraph of $H^2$ of $H^3$ and
      vertices $\beta'_{jk}$ such that
      $\{\gamma_{ij},\beta'_{jk},\gamma'_{ik}\}$ is an edge of $H^2$.
\item Lemma~\ref{lem:Rright} is used to obtain an $n'_1$-partitioned induced subhypergraph of $H^1$ of $H^2$ and
      vertices $\beta_{jk}$ such that
      $d_{jk,ik}(\beta_{jk},\gamma_{ik})\ge\varepsilon/2$ for all $1\le i<j<k\le n_1$.
\item Finally, Lemma~\ref{lem:Rleft} is used to obtain an $n$-partitioned induced subhypergraph of $H'$ of $H^1$ and
      vertices $\alpha_{ij}$ such that
      $\{\alpha_{ij},\beta_{jk},\gamma_{ik}\}$ is an edge of $H'$ for $1\le i<j<k\le n$.
\end{itemize}
The obtained hypergraph $H'$ together with the vertices $\alpha_{ij},\beta_{ij},\gamma_{ij},\beta'_{ij},\gamma'_{ij}$
satisfies the second case of the lemma.
The case of the third property in the statement of Lemma~\ref{lem:density} applies
is completely symmetric and yields the third case of the lemma.
\end{proof}

\section{Main theorem and examples}
\label{sec:final}

We are now ready to prove our main theorem.
This is done by transferring our result about $n$-partitioned hypergraphs contained in Lemma~\ref{lem:main}
back to the original setting of uniformly dense hypergraphs by using Proposition~\ref{prop:reiher}.  
The second and third properties in the statement of the theorem correspond to the cases of horizontal and vertical intersection, respectively,
as described in Lemmas~\ref{lem:density} and~\ref{lem:main};
note that the second and third cases in the two lemmas are symmetric (by reversing the order of the parts) and
so are associated with the case of vertical intersection.
The definition of a vanishing ordering can be found in Section~\ref{sec:intro}.

\begin{theorem}
\label{thm:main}
Let $H_0$ be an $n$-vertex $3$-graph that 
\begin{itemize}
\item has no vanishing ordering of its vertices,
\item can be partitioned into two spanning subhypergraphs $H_1$ and $H_2$ such that
      there exists an ordering of the vertices that is vanishing both for $H_1$ and $H_2$ and
      if $e_1$ is an edge of $H_1$ and $e_2$ is an edge of $H_2$ such that $|e_1\cap e_2|=2$,
      then the pair $e_1\cap e_2$ is right with respect to $H_1$ and left with respect to $H_2$, and
\item can be partitioned into two spanning subhypergraphs $H'_1$ and $H'_2$ such that
      there exists an ordering of the vertices that is vanishing both for $H'_1$ and $H'_2$ and
      if $e_1$ is an edge of $H'_1$ and $e_2$ is an edge of $H'_2$ such that $|e_1\cap e_2|=2$,
      then the pair $e_1\cap e_2$ is top with respect to $H'_1$ and left with respect to $H'_2$.
\end{itemize}
The uniform Tur\'an density of $H_0$ is equal to $1/27$.
\end{theorem}

\begin{proof}
Fix an $n$-vertex $3$-graph $H_0$ with the properties given in the statement of the lemma.
Since $H_0$ has no vanishing ordering, its uniform Tur\'an density is at least $1/27$ by Theorem~\ref{thm:rrs} and Corollary~\ref{cor:rrs}.
By Proposition~\ref{prop:reiher}, we need to show that for every $\delta>0$,
there exists $N$ such that every $N$-partitioned hypergraph with density at least $1/27+\delta$ embeds $H_0$.
Apply Lemma~\ref{lem:main} with $n$ and $\delta$ to get $N$.
Let $H$ be an $N$-partitioned hypergraph with density at least $1/27+\delta$ and
let $H'$ be an $n$-partitioned induced subhypergraph of $H$ with one of the three properties given in Lemma~\ref{lem:main}.
In each of the three cases given by which of the three properties holds,
we use a partition given in the second or in the third case of the statement of the theorem,
to embed $H_1$ using the edges formed by the vertices $\alpha_{ij}$, $\beta_{ij}$ and $\gamma_{ij}$, and
to embed $H_2$ using the edges formed by the vertices $\beta'_{ij}$, $\gamma'_{ij}$ and either $\beta_{ij}$ or $\gamma_{ij}$.

If the first property given in Lemma~\ref{lem:main} holds,
we consider an ordering of the vertices of $H_0$ as described in the second bullet point in the statement of the theorem and
choose $\alpha_{ij}$ for every pair $i,j$ that is left with respect to $H_1$,
choose $\beta_{ij}$ for every pair $i,j$ that is right with respect to $H_1$ or left with respect to $H_2$,
choose $\gamma_{ij}$ for every pair $i,j$ that is top with respect to $H_1$,
choose $\beta'_{ij}$ for every pair $i,j$ that is right with respect to $H_2$, and
choose $\gamma'_{ij}$ for every pair $i,j$ that is top with respect to $H_2$.
Hence, $H'$ embeds $H_0$.

If the second property given in Lemma~\ref{lem:main} holds,
we consider an ordering of the vertices of $H_0$ as described in the third bullet point in the statement of the theorem and
choose $\alpha_{ij}$ for every pair $i,j$ that is left with respect to $H_1$,
choose $\beta_{ij}$ for every pair $i,j$ that is right with respect to $H_1$,
choose $\gamma_{ij}$ for every pair $i,j$ that is top with respect to $H_1$ or left with respect to $H_2$,
choose $\beta'_{ij}$ for every pair $i,j$ that is right with respect to $H_2$, and
choose $\gamma'_{ij}$ for every pair $i,j$ that is top with respect to $H_2$.
Again, we conclude that $H'$ embeds $H_0$.

The third property given in Lemma~\ref{lem:main} is symmetric to the second (by reversing the order of the parts of $H'$), and
the arguments as in the previous paragraph yield that $H'$ embeds $H_0$.
Since $H'$ embeds $H_0$ regardless which of the three properties given in Lemma~\ref{lem:main} holds,
the uniform Tur\'an density of $H_0$ is at most $1/27$ by Proposition~\ref{prop:reiher}.
\end{proof}

We next give examples of $3$-graphs that satisfy the assumption of Theorem~\ref{thm:main} and
so their uniform Tur\'an density is equal to $1/27$.
We start with introducing a lemma,
which will be useful to rule out the existence of a vanishing ordering of vertices of a $3$-graph.
We say that a directed graph is \emph{simple} if every pair of its vertices is joined by at most one edge,
i.e., there are no parallel edges or pairs of edges oriented in the opposite way.

\begin{lemma}
\label{lem:vanishing}
A $3$-graph $H$ has a vanishing ordering if and only if
there exists a simple directed graph $G$ with the same vertex set as $H$ such that
each edge of $H$ corresponds to a cyclically directed triangle with edges colored $1$, $2$ and $3$ (in this order), and
there exist distinct indices $i$ and $j$, $i,j\in\{1,2,3\}$, such that
the subgraph of $G$ containing all edges colored with $i$ and $j$ is acyclic.
\end{lemma}

\begin{proof}
We show that the existence of a directed graph $G$ with properties as given in the statement of the lemma
is equivalent to the existence of a vanishing ordering of $H$.
First suppose that there exists a directed graph $G$ with edges colored as described in the statement of the lemma.
By symmetry, we may assume that the subgraph of $G$ containing all edges colored with $1$ and $2$ is acyclic (otherwise,
we cyclically rotate the colors to satisfy this).
Consider a linear ordering of the vertices of $H$ that
is an extension of the partial order given by the existence of a directed path in $G$.
We claim that this linear ordering is vanishing.
Indeed, all left edges are colored with $1$, all right edges with $2$ and all top edges with $3$.

Next suppose that there exists a vanishing ordering of the vertices of $H$ and
consider the following simple directed graph $G$:
if $\{u,v,w\}$ is an edge of $H$ such that $uv$ is the left pair, $vw$ is the right pair and $uw$ is the top pair,
include the edge $uv$ directed from $u$ to $v$ and colored with $1$,
the edge $vw$ directed from $v$ to $w$ and colored with $2$, and
the edge $uw$ directed from $w$ to $u$ and colored with $3$.
The subgraph of $G$ containing all edges colored with $1$ and $2$ satisfies that
every edge is directed from a smaller vertex to a larger vertex, and so the subgraph is acyclic.
Hence, $G$ has the properties described in the statement of the lemma.
\end{proof}

As the first example of a $3$-graph with uniform Tur\'an density equal to $1/27$,
we present a $3$-graph with seven vertices, which is the smallest possible number of vertices.
The $3$-graph has a non-trivial group of automorphisms,
which correspond to a vertical mirror symmetry in Figure~\ref{fig:example9a} where the $3$-graph is visualized.

\begin{theorem}
\label{thm:example9}
Let $H$ be a $3$-graph with seven vertices $a,\ldots,g$ and
the following $9$ edges: $abc$, $ade$, $bcd$, $bcf$, $cde$, $def$, $abg$, $cdg$ and $efg$.
The uniform Tur\'an density of $H$ is equal to $1/27$.
\end{theorem}

\begin{figure}
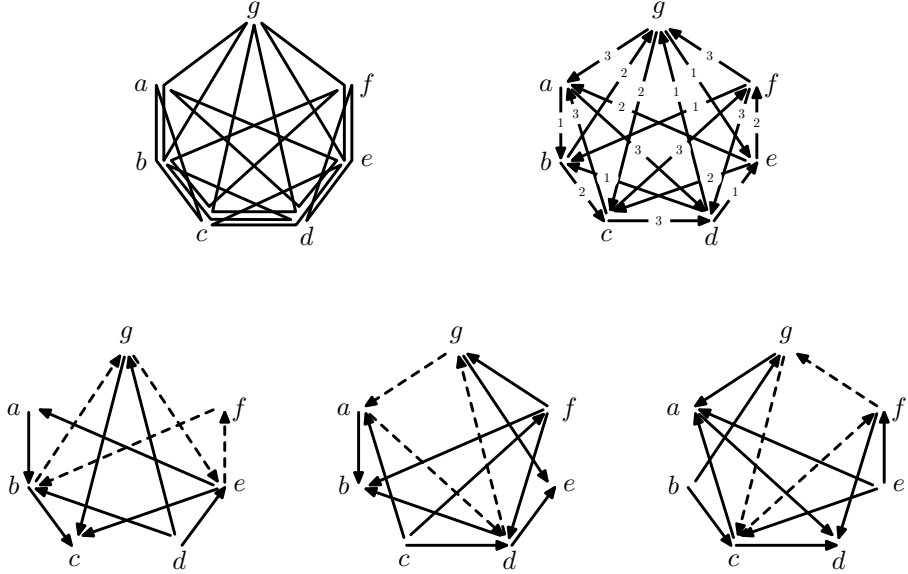

\begin{center}
\epsfbox{turan27-1.mps}
\hskip 2cm
\epsfbox{turan27-2.mps}
\vskip 1cm
\epsfbox{turan27-3.mps}
\hskip 1cm
\epsfbox{turan27-4.mps}
\hskip 1cm
\epsfbox{turan27-5.mps}
\end{center}
\caption{The $3$-graph $H$ described in the statement of Theorem~\ref{thm:example9} (the edges correspond to the drawn triangles),
         the unique (up to a symmetry) graph $G$ associated with $H$ as described in Lemma~\ref{lem:vanishing} and
	 the three subgraphs containing all edges with distinct pairs of colors.
	 Cycles witnessing that neither of the three subgraphs is acyclic are drawn dashed.}
\label{fig:example9a}
\end{figure}

\begin{figure}
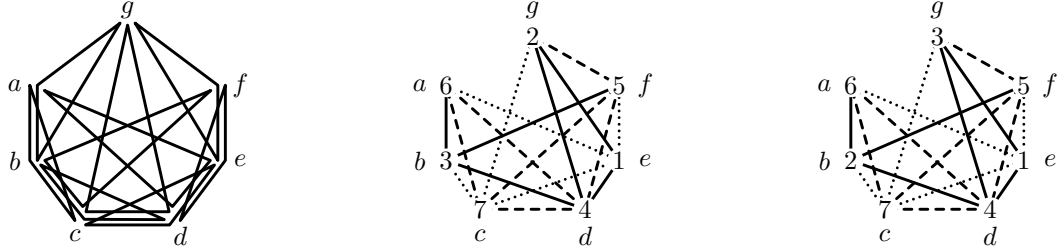

\begin{center}
\epsfbox{turan27-1.mps}
\hskip 2cm
\epsfbox{turan27-6.mps}
\hskip 2cm
\epsfbox{turan27-7.mps}
\end{center}
\caption{The $3$-graph $H$ described in the statement of Theorem~\ref{thm:example9} and
         the vanishing orders with respect to $3$-graphs $H_2$ and $H'_2$ as in the statement of Theorem~\ref{thm:main};
	 the $3$-graphs $H_2$ and $H'_2$ are obtained by removing the edge $abg$ from $H$.
	 The left pairs are drawn solid, the right pairs dashed and the top pairs dotted.}
\label{fig:example9b}
\end{figure}

\begin{proof}
Consider a directed graph $G$ associated with the $3$-graph $H$ as described in the statement of Lemma~\ref{lem:vanishing}.
Note that if we fix an orientation and a coloring for the triple $abc$ then all the orientations and colors of the graph are fixed. Therefore $G$ is unique up to cyclical shifts of the colors and a swap of all the orientations. Hence it suffices to consider the directed graph $G$ depicted in Figure~\ref{fig:example9a} together with the three subgraphs
containing all edges of the colors $1$ and $2$, all edges of the colors $1$ and $3$, and all edges of the colors $2$ and $3$.
Since neither of the three subgraphs is acyclic,
the $3$-graph $H$ has no vanishing ordering by Lemma~\ref{lem:vanishing}.
Hence, the first condition in the statement of Theorem~\ref{thm:main} holds.

We next verify the second and third conditions in the statement of Theorem~\ref{thm:main}.
We set $H_1$ and $H'_1$ to be the $3$-graphs with the same vertex set as $H$ and the edge $abg$ only, and
$H_2$ and $H'_2$ the $3$-graphs obtained from $H$ by removing the edge $abg$.
We consider the ordering $egbdfac$ of the vertices of $H_2$ and the ordering $ebgdfac$ of the vertices of $H'_2$.
The orderings are vanishing with respect to $H_2$ and $H'_2$, respectively;
this can be straightforwardly verified with the aid of Figure~\ref{fig:example9b}.
Note that $ab$ is the only pair shared by an edge of both $H_1$ and $H_2$, as well as the only pair shared by an edge of both $H_1'$ and $H_2'$.
The pair $ab$ is left with respect to both $H_2$ and $H'_2$.
Since the pair $ab$ is right with respect to $H_1$ and top with respect to $H'_1$ (the orderings are vanishing
with respect to $H_1$ and $H'_1$ as the $3$-graph consists of a single edge),
the second and third conditions in the statement of Theorem~\ref{thm:main} hold.
We conclude that the uniform Tur\'an density of $H$ is equal to $1/27$.
\end{proof}

We next present an infinite family of $3$-graphs with uniform Tur\'an density equal to $1/27$;
the smallest $3$-graph in the family has eight vertices.
The family enjoys three cyclic symmetries (by mapping the vertices $c_i$, $d_i$ and $e_i$ to each other in a cyclic way).

\begin{theorem}
\label{thm:example8}
For a positive integer $k$,
let $H^k$ be the $3$-graph with $5+3k$ vertices $a,b,c_0,\ldots,c_k,d_0,\ldots,d_k,e_0,\ldots e_k$ and
the following $3(k+2)$ edges:
\[
\begin{array}{llllll}
abc_0, & bc_0c_1, & c_0c_1c_2, & \ldots, & c_{k-2}c_{k-1}c_k, & c_{k-1}c_kd_k,\\
abd_0, & bd_0d_1, & d_0d_1d_2, & \ldots, & d_{k-2}d_{k-1}d_k, & d_{k-1}d_ke_k,\\
abe_0, & be_0e_1, & e_0e_1e_2, & \ldots, & e_{k-2}e_{k-1}e_k, & e_{k-1}e_kc_k.
\end{array}
\]
The uniform Tur\'an density of $H^k$ is equal to $1/27$.
\end{theorem}

\begin{figure}
\begin{center}
\epsfbox{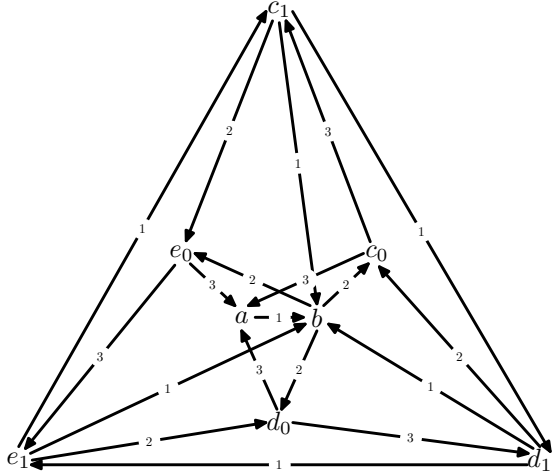}
\end{center}
\caption{The graph $G$ from the proof of Theorem~\ref{thm:example8} for $k=1$.}
\label{fig:example8a}
\end{figure}

\begin{proof}
Fix a positive integer $k$.
We first show using Lemma~\ref{lem:vanishing}, Theorem~\ref{thm:rrs} and Corollary~\ref{cor:rrs} that the uniform Tur\'an density of $H^k$ is at least $1/27$.
Consider the graph $G$ as described in the statement of the lemma.
By symmetry,
we can assume that the edge $ab$ is oriented from $a$ to $b$ and colored with $1$ (see Figure~\ref{fig:example8a} for $k=1$).
So, the edge $bx_0$ is oriented from $b$ to $x_0$ and colored with $2$ for each $x\in\{c,d,e\}$, and
the edge $x_0x_1$ is oriented from $x_0$ to $x_1$ and colored with $3$, etc.
In particular, the edge $x_{k-1}x_k$ is oriented from $x_{k-1}$ to $x_k$ and colored with $k+2\pmod 3$.
It follows that
the edges $c_kd_k$, $d_ke_k$ and $e_kc_k$ form a cyclically oriented triangle and each of the edges is colored with $k\pmod 3$, and
the edges $c_{k-1}c_k$, $c_ke_{k-1}$, $e_{k-1}e_k$, $e_kd_{k-1}$, $d_{k-1}d_k$ and $d_kc_{k-1}$
form an oriented cycle with edges colored with $k+2\pmod 3$ and $k+1\pmod 3$ in an alternating way.
Hence, no pair of edge colors induces an acyclic subgraph, and
so the $3$-graph $H^k$ has no vanishing ordering of the vertices by Lemma~\ref{lem:vanishing}.
We conclude that the uniform Tur\'an density of $H^k$ is at least $1/27$.

We next verify the second and third conditions in the statement of Theorem~\ref{thm:main}.
We set $H_1$ and $H'_1$ to be the $3$-graphs with the same vertex set as $H^k$ and the edge $e_{k-1}e_kc_k$ only, and
$H_2$ and $H'_2$ the $3$-graphs obtained from $H^k$ by removing the edge $e_{k-1}e_kc_k$.
Let $A$ be the set containing all vertices $x_i$ with $i\equiv k-1\pmod 3$, $x\in\{c,d,e\}$,
$a$ if $k\equiv 2\pmod 3$, and $b$ if $k\equiv 0\pmod 3$.
Let $B$ be the set containing all vertices $x_i$ with $i\equiv k\pmod 3$, $x\in\{c,d,e\}$, except for $c_k$ and $d_k$,
$a$ if $k\equiv 1\pmod 3$, and $b$ if $k\equiv 2\pmod 3$.
Finally, let $C$ be the set containing all vertices $x_i$ with $i\equiv k+1\pmod 3$, $x\in\{c,d,e\}$,
$a$ if $k\equiv 0\pmod 3$, and $b$ if $k\equiv 1\pmod 3$.

First consider any ordering of the vertices of $H^k$ that contains first all vertices of $A$ except for $e_{k-1}$,
then $c_k$, then $e_{k-1}$, then $d_k$, then all vertices of $B$, and then all vertices of $C$.
Observe that this ordering is a vanishing ordering with respect to $H'_1$ and
the pair $e_{k-1}e_k$ is left in this ordering.
Indeed, each edge of $H'_1$ except for those containing the vertex $c_k$ or the vertex $d_k$,
i.e., except for $c_{k-2}c_{k-1}c_k$, $d_{k-2}d_{k-1}d_k$, $c_{k-1}c_kd_k$ and $d_{k-1}d_ke_k$,
contains exactly one vertex from $A$, one from $B$ and one from $C$, and
so the pairs involving a vertex from $A$ and a vertex from $B$ except the pair $d_{k-1}e_k$ are left,
the pairs involving a vertex from $A$ and a vertex from $C$ are top, and
the pairs involving a vertex from $B$ and a vertex from $C$ are right;
in addition, the pairs $c_{k-1}c_k$ and $d_{k-1}d_k$ are left, 
the pairs $c_kc_{k-2}$, $d_kd_{k-2}$ and $d_ke_k$ are right, and
the pairs $c_{k-1}d_k$ and $d_{k-1}e_k$ are top,
which is in line with the ordering of the four exceptional edges.
Since $H_1$ contains a single edge (the edge $e_{k-1}e_kc_k$),
the ordering is also a vanishing ordering with respect to $H_1$. 
Furthermore, $e_{k-1}e_k$ is the only pair shared by $H_1$ and $H'_1$ and
the pair right with respect to this ordering in $H_1$ and left in $H_1'$.
Hence, the second condition in the statement of Theorem~\ref{thm:main} holds.

Next consider any ordering of the vertices of $H^k$ that contains first all vertices of $A$,
then $c_k$, then $d_k$, then all vertices of $B$, and then all vertices of $C$.
Observe that this ordering is a vanishing ordering with respect to $H'_2$ and
the pair $e_{k-1}e_k$ is left in this ordering (the argument is analogous to the previous case).
Since $H_2$ contains a single edge, the ordering is also a vanishing ordering with respect to $H_2$ and
the pair $e_{k-1}e_k$ is top.
We conclude that the third condition in the statement of Theorem~\ref{thm:main} also holds, and
so the uniform Tur\'an density of $H$ is equal to $1/27$.
\end{proof}

\section{Conclusion}

The $7$-vertex $3$-graph with uniform Tur\'an density $1/27$ described in Theorem~\ref{thm:example9}
has the smallest possible number of vertices
but it is not the unique $7$-vertex $3$-graph with uniform Tur\'an density equal to $1/27$.
Using a computer, we have generated all minimal $7$-vertex $3$-graphs
with uniform Tur\'an density equal to $1/27$ and
we include their list below (three of them have one fewer edge than the $3$-graph from Theorem~\ref{thm:example9},
however, they enjoy less symmetries than the presented $3$-graph and
so we preferred analyzing a more symmetric $3$-graph with a larger number of edges).
The vertices are denoted by $a,\ldots,g$ and each line below is the edge set of one of them;
the first line contains the $3$-graph described in Theorem~\ref{thm:example9} (with vertices renamed).

\begin{align*}
& abc, abd, abe, acf, acg, bdf, bdg, cef, deg\\
& abc, abd, abe, acf, acg, bdf, cdg, cef, efg\\
& abc, abd, abe, acf, adg, bdf, cef, efg\\
& abc, abd, abe, acf, aeg, bdf, bfg, cde, cdg, cef\\
& abc, abd, abe, acf, bcg, bdf, cde, ceg, efg\\
& abc, abd, ace, adg, bcf, bde, bfg, cdf, ceg\\
& abc, abd, ace, aef, afg, bcf, bde, beg, cdf, cdg\\
& abc, abd, ace, afg, bcf, bde, bfg, def\\
& abc, abd, ace, bde, bfg, cdf, ceg, cfg
\end{align*}

For each of the remaining $15$ minimal $7$-vertex $3$-graphs $H$ with positive uniform Tur\'an density (out of which $6$ have isolated vertices),
for every $\varepsilon>0$,
there exist arbitrarily large $(4/27,\varepsilon)$-dense $3$-graphs that avoid $H$.
Each of these $3$-graphs $H$ is avoided by one of the following two constructions of random $n$-vertex $3$-graphs.
Order the $n$ vertices randomly and color the pairs of vertices randomly red and blue with probability $2/3$ and $1/3$, respectively.
In the first construction, we include an edge if the left and right pairs are red and the top pair is blue, and
in the second construction, we include an edge if the left and top pairs are red and the right pair is blue.

Theorem~\ref{thm:main} gives a sufficient condition on a $3$-graph to have the uniform Tur\'an density equal to $1/27$.
We believe that this condition is not necessary, however,
we do not have an example of a $3$-graph with uniform Tur\'an density $1/27$ that does not satisfy the condition and
do not also have a conjecture for a possible classification of $3$-graphs with uniform Tur\'an density $1/27$.

\begin{problem}
Characterize the $3$-graphs with uniform Tur\'an density equal to $1/27$.
\end{problem}

In view of Corollary~\ref{cor:rrs},
it is natural to ask whether a similar phenomenon appears for the uniform Tur\'an density of $1/27$,
in particular, all $3$-graphs that we know to fail to satisfy the conditions of Theorems~\ref{thm:rrs} and~\ref{thm:main}
have uniform uniform Tur\'an density at least $4/27$.

\begin{problem}
Does there exist $\delta>0$ such that the uniform Tur\'an density of every $3$-graph is either at most $1/27$ or at least $1/27+\delta$?
\end{problem}

\section*{Acknowledgement}

The authors would like to thank Jacob Cooper for his comments on the topics covered in this paper. We would also like to thank the anonymous referee for their careful review of the manuscript and for their many helpful comments.

\bibliographystyle{bibstyle}
\bibliography{turan27}

\end{document}